\documentclass[12pt]{amsart}

\usepackage[latin1]{inputenc}
\usepackage{amsmath,amsthm,amsfonts,amscd,amssymb,xspace}
\usepackage{verbatim}
\usepackage{graphicx,tikz}

\newcommand{\td}[2]{\xrightarrow[#1\rightarrow #2]{}}

\def\noi{\noindent}

\newtheorem{Thm}{Theorem}[section]
\newtheorem{Prop}[Thm]{Proposition}
\newtheorem{Lem}[Thm]{Lemma}
\newtheorem{Cor}[Thm]{Corollary}
\newtheorem{Fact}[Thm]{Fact}
\newtheorem{Def}[Thm]{Definition}

\numberwithin{equation}{section}

\newcommand{\g}[1]{{\mbox{\goth #1}}}
\newcommand{\m}[1]{\mathbb{ #1}}

     \def\ol{\overline}    
   \def\wt{\widetilde}
\def\al{\alpha}               
       \def\eps{\varepsilon}  
\def\th{\theta}       

\def\ph{\varphi}               \def\ps{\psi}

                  \def\Om{\Omega}

\theoremstyle{definition}

\theoremstyle{remark}

\newtheorem{Rmq}[Thm]{Remark}
\newtheorem{Ex}[Thm]{Example}

\numberwithin{equation}{section}
\newfont{\goth}{eufm10 at 12pt}
\newfont{\gots}{eufm8 at 9pt}

\def\bt{\begin{Thm}}
\def\et{\end{Thm}}
\def\br{\begin{Rmq}}
\def\er{\end{Rmq}}

\def\bc{\begin{Cor}}
\def\ec{\end{Cor}}
\def\bp{\begin{Prop}}
\def\ep{\end{Prop}}
\def\bl{\begin{Lem}}
\def\el{\end{Lem}}
\def\bd{\begin{Def}}
\def\ed{\end{Def}}
\def\bq{\begin{quotation}}
\def\eq{\end{quotation}}
\def\bfa{\begin{Fact}}
\def\efa{\end{Fact}}

\def\ra{\rightarrow}
\def\vs{\vspace{1em}}
\def\D{\displaystyle}

\begin{document}

\title[p-adic Lie groups]{How far are p-adic Lie groups from algebraic groups?}

\author{Yves Benoist and Jean-Fran\c cois Quint}

\address{CNRS -- Universit\'e Paris-Saclay}

\email{yves.benoist@math.u-psud.fr}

\address{CNRS -- Universit\'e Bordeaux}

\email{jquint@math.u-bordeaux1.fr}
\thanks{}

\maketitle


\begin{abstract}{
We show that, in a weakly regular $p$-adic Lie group $G$,
the subgroup $G_u$ spanned by the 
one-parame\-ter subgroups of $G$
admits a Levi decomposition.
As a consequence, 
there exists a regular open subgroup of $G$ which contains $G_u$}
\end{abstract}

\renewcommand{\thefootnote}{\fnsymbol{footnote}} 
\renewcommand{\thefootnote}{\arabic{footnote}} 


{\footnotesize \tableofcontents}


\section{Introduction}
\label{secintro}

\subsection{Motivations}
\label{secmot}

When studying the dynamics of the  subgroups 
of a $p$-adic Lie group $G$ on its homogeneous spaces,
various assumptions can be made on $G$. 
For instance, one can ask $G$ to be algebraic as in \cite{MaTo} 
i.e. to be a subgroup of a linear group defined by polynomial equations. 

Another possible assumption on $G$ is 
regularity.
One asks $G$ to satisfy properties that are well-known to hold 
for Zariski connected algebraic $p$-adic Lie groups: 
a uniform bound on the cardinality of the finite subgroups and 
a characterization of the center as the kernel of the adjoint representation
(see Definition \ref{defregsad}).
Ratner's theorems in \cite{Ra3} are written under this regularity assumption.

A  more natural and weaker assumption on $G$ in this context
is the weak regularity of $G$
i.e. the fact that the one-parameter morphisms of $G$ 
are uniquely determined by their derivative
(see Definition \ref{defweareg}).
Our paper \cite{BQII} is written under this weak-regularity assumption.

The aim of this text is to clarify the relationships between these three 
assumptions.

A key ingredient in the proof is  Proposition \ref{proprarag}. It claims\\
\centerline{\small{\it the finiteness of the center of 
the universal topological}}
\centerline{\small{\it central extension of non-discrete
simple $p$-adic Lie groups.}}

This is a fact which is due 
to Prasad and Raghunathan in \cite{PR84}.
\vs

\subsection{Main results}
\label{secmaires}
We will first prove (Proposition \ref{prolevdec}):
\bq
{\it In a weakly regular $p$-adic Lie group $G$, 
the subgroup $G_u$ spanned by the one-parameter subgroups
of $G$ is closed and  
admits a Levi decomposition},
\eq
i.e. $G_u$ is a semidirect product $G_u=S_u\ltimes R_u$ of a 
group $S_u$ by a normal subgroup $R_u$ where $S_u$ is a
finite cover  of a finite index subgroup of an algebraic semisimple
Lie group and where $R_u$ is algebraic unipotent.

As a consequence, we will prove (Theorem \ref{thwearegreg}):
\bq
{\it In a weakly regular $p$-adic Lie group $G$ 
there always exists an open subgroup $G_\Om$ which is regular and 
contains  $G_u$.}  
\eq
This Theorem \ref{thwearegreg}
is useful since it  extends the level of generality of Ratner's theorems in \cite{Ra3}
(see \cite[Th 5.15]{BQII}).
More precisely, Ratner's  theorems 
for products of real and $p$-adic Lie groups in \cite{Ra3} 
are proven under the assumption that these $p$-adic Lie groups are {\it regular}. 
Ratner's theorems can be extended under the weaker assumption that
these $p$-adic Lie groups are
{\it weakly regular} thanks to our
Theorem \ref{thwearegreg}.

This Theorem \ref{thwearegreg} has been announced 
in \cite[Prop. 5.8]{BQII} and has been used in the same paper.

The strategy consists in proving first 
various statements for weakly regular 
$p$-adic Lie groups which were proven in \cite{Ra3}
under the regularity assumption.
To clarify the discussion we will reprove also
the results from \cite{Ra3} that we need. 
But we will take for granted classical results 
on the structure of simple $p$-adic algebraic groups 
that can be found in \cite{BoTi65}, \cite{BT1}, \cite{BT2} or \cite{PlRa94}.   

\subsection{Plan}
\label{secplan}

In the preliminary Chapter \ref{secpad}  we recall
a few definitions and examples. 

Our main task  in Chapters \ref{secalguni} and 
\ref{secderone} is to describe, for a weakly regular $p$-adic Lie group
$G$, the subset 
$\g g_G\subset \g g$ of   derivatives of  one-parameter morphisms of $G$.

In Chapter \ref{secalguni}, the results are mainly due to Ratner.
We first study the nilpotent $p$-adic Lie subgroups $N$ of $G$ 
spanned by one-parameter subgroups. We will see that they satisfy
$\g n\subset \g g_G$ (Proposition \ref{prongggn}). 
Those $p$-adic Lie groups $N$ are called algebraic unipotent.
We will simultaneously compare this set $\g g_G$ with the analogous set $\g g'_{G'}$ 
for a quotient group $G'=G/N$ of $G$ when $N$ is a normal 
algebraic unipotent subgroup (Lemma \ref{lemngq}).
This will allow us to prove that $G$ contains a largest normal
algebraic unipotent subgroup (Proposition \ref{proruru}).

In Chapter 
\ref{secderone}, we will then be able to describe precisely the set $\g g_G$
using a Levi decomposition of $\g g$
(Proposition \ref{proxxsxr}).
A key ingredient is a technic, borrowed from \cite{BQII}, for
constructing  one-parameter subgroups 
in a $p$-adic Lie group $G$
(Lemma \ref{lemconstrunipo}).

In the last Chapter \ref{secgrospauni}, 
we will prove the main results we have just stated: 
Proposition \ref{prolevdec}
and Theorem \ref{thwearegreg},
using Prasad--Raghunathan finiteness theorem in \cite{PR84}
(see Proposition \ref{proprarag}).
We will end this text by an example showing that,
when a $p$-adic Lie group $G$ with $G=G_u$ 
is not assumed to be weakly  regular, 
it does not always contain 
a regular open subgroup $H$ for which $H=H_u$
(Example \ref{exagguhhu}).

\section{Preliminary results}
\label{secpad}

\bq
We recall here a few definitions and results from 
\cite{Ra3}.
\eq

Let $p$ be a prime number and $\m Q_p$ be the field of $p$-adic numbers.
When $G$ is a $p$-adic Lie group (see \cite{Bou1}),
we will always denote
by $\g g$ the Lie algebra of $G$. It is a $\m Q_p$-vector space.
We will denote by ${\rm Ad}_{\g g}$ or ${\rm Ad}$ 
the adjoint action of  $G$ on $\g g$
and by ${\rm ad}_{\g g}$ or ${\rm ad}$ 
the adjoint action of  $\g g$ on $\g g$.
Any closed subgroup $H$ of $G$ is a $p$-adic Lie subgroup 
and its Lie algebra
$\g h$ is a $\m Q_p$-vector subspace of $\g g$ (see \cite[Prop. 1.5]{Ra3}).

We choose a ultrametric norm $\|.\|$ on  $\g g$ with values in $p^{\m Z}$.

\subsection{One-parameter subgroups}
\label{seconepar}
\bd
\label{defonepar}
A {\it one-parameter morphism} $\ph$ of a $p$-adic Lie group $G$ is a continuous morphism
$\ph:\m Q_p\ra G;t\mapsto \ph(t)$.
A {\it one-parameter} subgroup  is the image $\ph(\m Q_p)$ of an injective
one-parameter morphism. 
\ed
The derivative of a one-parameter morphism is 
an element $X$ of $\g g$ for which ${\rm ad}X$ is nilpotent. 
This follows from the following Lemma from \cite[Corollary 1.2]{Ra3}.

\bl
\label{lemexpgl}
Let $\ph:\m Q_p\ra {\rm GL}(d,\m Q_p)$ be a one-parameter morphism. 
Then there exists a nilpotent matrix $X$ in ${\g g\g l}(d,\m Q_p)$
such that $\ph(t)={\rm exp}(t X)$ for all $t$ in $\m Q_p$.
\el

Here ${\rm exp}$ is the exponential of matrices~:\; 
${\rm exp}(t X):=\sum_{n\geq 0}\tfrac{t^nX^n}{n!}$

\begin{proof} First, we claim that, if $K$ is a finite extension of $\mathbb Q_p$, any continuous 
one-parameter morphism $\psi:\mathbb Q_p\rightarrow K^*$ is constant. Indeed, 
since the modulus of $\psi$ is a continuous morphism form $\mathbb Q_p$ to a discrete multiplicative 
subgroup of $(0,\infty)$, the kernel of $|\psi|$ contains $p^k\mathbb Z_p$, for some integer $k$. 
Since $\mathbb Q_p/p^k\mathbb Z_p$ is a torsion group and $(0,\infty)$ has no torsion, $\psi$ has constant modulus,
that is, $\psi$ takes values in $\mathcal O^*$, where $\mathcal O$ is the integer ring of $K$. Now, on one hand, $\mathcal O^*$ is a profinite group, that is, 
it is a compact totally discontinuous group and hence it admits a basis of neighborhoods of the identity which are finite index subgroups (namely, for example, the subgroups $1+p^k\mathcal O$, $k\geq 0$). 
On the other hand, every closed subgroup of $\mathbb Q_p$ is of the form $p^k\mathbb Z_p$, for some integer $k$, and hence, has infinite index in $\mathbb Q_p$ and therefore any 
continuous morphism from $\mathbb Q_p$ to a finite group is trivial. Thus, $\psi$ is constant, which should be proved.

Let now $\varphi$ be as in the lemma and let
$X\in {\g g\g l}(d,\m Q_p)$ be the derivative of $\ph$.
After having simultaneously reduced the commutative family of matrices $\ph(t)_{t\in \m Q_p}$,
the joint eigenvalues give continuous morphisms $\m Q_p\ra K^*$ 
where $K$ is a finite extension of $\m Q_p$. By the remark above, these morphisms are constant, that is, 
there exists $g$ in ${\rm GL}(d,\mathbb Q_p)$, such that, for any $t$ in $\mathbb Q_p$, the matrix $g\varphi(t)g^{-1}$ is 
unipotent and upper triangular. We may assume $g=1$.
Then $X$ is  nilpotent and upper triangular and it remains to check that the map 
$\th:t\mapsto\ph(t){\rm exp}(-tX)$ is constant.
Since $\ph(t)$ commutes with $X$, this map $\th$ 
is a one-parameter morphism with zero derivative. 
Hence, the kernel of $\psi$ is an open subgroup and the matrices $\th(t)$, $t\in\mathbb Q_p$, have finite order.
Since they are unipotent, they equal  $e$, which should be proved.
\end{proof}

\subsection{Weakly regular and regular $p$-adic Lie groups}
\label{secnot}

\begin{Def}[Ratner, \cite{Ra3}]
\label{defweareg}
A $p$-adic Lie group $G$ is said to be {\it weakly regular}
if any two one-parameter morphisms $\m Q_p\ra G$
with the same derivative  are equal.
\end{Def}

Note that any closed subgroup of a weakly regular $p$-adic Lie group
is also weakly regular.

\begin{Ex}[{\cite[Cor. 1.3 and Prop. 1.5]{Ra3}}]
Every closed subgroup
of 
$ {\rm GL}(d,\m Q_p)$ is
weakly regular.
\end{Ex}

\begin{proof}
This follows from Lemma \ref{lemexpgl}.
\end{proof}

\begin{Def}[Ratner, \cite{Ra3}]
\label{defregsad}
A $p$-adic Lie group $G$ is said to be {\rm Ad}-{\it regular}
if the kernel of the adjoint map ${\rm Ker}({\rm Ad}_\g g)$
is equal to the center $Z(G)$ of $G$. 
It is said to be {\it regular} if it is {\rm Ad}-{\it regular} and if the
finite subgroups of $G$ have uniformly bounded cardinality.
\end{Def}

Note that any open subgroup of a regular $p$-adic Lie group
is also regular.

This definition is motivated by the following example.

\begin{Ex}
\label{exalinreg}
$a)$ The finite subgroups of a compact $p$-adic Lie group $K$
have uniformly bounded cardinality.

$b)$ The finite subgroups of a $p$-adic linear group 
have uniformly bounded cardinality.

$c)$ The Zariski connected linear algebraic $p$-adic Lie groups $G$ are
regular.
\end{Ex}

\begin{proof}[Proof of Example \ref{exalinreg}] (see \cite{Ra3})

$a)$ Since $K$ contains a torsion free open normal subgroup
$\Om$, for every finite subgroup $F$ of $K$, 
one has the bound $|F|\leq |K/\Om|$.

$b)$ We want to bound the cardinality $|F|$ of a finite subgroup of
a group $G\subset {\rm GL}(d,\m Q_p)$. This follows from $a)$ since 
$F$ is included in a conjugate of the compact group 
$K= {\rm GL}(d,\m Z_p)$. 

$c)$ It remains to check that $G$ is {\rm Ad}-regular. 
Let $g$ be an element in the kernel of the adjoint map ${\rm Ad}_{\g g}$.
This means that the centralizer $Z_g$ of $g$ in $G$ is an open subgroup of $G$.
This group $Z_g$ is also Zariski closed. Hence it is Zariski open. 
Since $G$ is Zariski connected, one deduces $Z_g=G$ and 
$g$ belongs to the center of $G$.
\end{proof}

We want to relate the  two notions ``weakly regular'' and ``regular''.
We first recall the following implication in \cite[Cor. 1.3]{Ra3}.
\bl
\label{lemregweareg}
Any
regular $p$-adic Lie group is weakly regular. 
\el

\begin{proof} 
Let $\ph_1:\m Q_p\ra G$ and $\ph_2:\m Q_p\ra G$ 
be one-parameter morphisms of $G$ with the same derivative. 
We want to prove that $\ph_1=\ph_2$. 
According to Lemma \ref{lemexpgl}, the one-parameter morphisms 
${\rm Ad}_{\g g}\ph_1$ and ${\rm Ad}_{\g g}\ph_2$ are equal. 
Since $G$ is ${\rm Ad}$-regular, this implies that, for all $t$ in $\m Q_p$, 
the element $\th(t):=\ph_1(t)^{-1}\ph_2(t)$ is in the center of $G$. 
Hence $\th$ is a one-parameter morphism of the center of $G$ 
with zero derivative. Its image is either trivial or an infinite $p$-torsion group.
This second case is excluded by the uniform bound on the finite subgroups of $G$.
This proves the equality  $\ph_1=\ph_2$.
\end{proof}

The aim of this text is to prove  Theorem 
\ref{thwearegreg} which is a kind of converse to Lemma \ref{lemregweareg}.

\section{Algebraic unipotent $p$-adic Lie group}
\label{secalguni}
\bq
In this chapter, we study the algebraic unipotent subgroups
of a weakly regular $p$-adic Lie group. 
The results in this chapter are mainly due to Ratner.
\eq
 
\subsection{Definition and closedness}
\label{secdefclo}
\bq 
We first focus on a special class of $p$-adic Lie groups.
\eq

We say that an element $g$ of a $p$-adic Lie group $G$
admits a logarithm if one has $g^{p^n}\td{n}{\infty}e$: indeed,
for such a $g$, the map $n\mapsto g^n$ extends as a continuous
morphism $\m Z_p\ra G$ and one can define the logarithm $\log (g)\in \g g$ 
as being the derivative at $0$ of this morphism. 

\bd
\label{defuni}
A $p$-adic Lie group $N$ is called algebraic unipotent
if its Lie algebra is nilpotent,
if every element  $g$ in $N$ admits a logarithm $\log (g)$, and if  
the logarithm map $\log:N\ra\g n$ is a bijection.
\ed

This implies that every non trivial element $g$ of $N$ 
belongs to a unique one-parameter subgroup of $N$.
By definition these groups $N$ are weakly regular.
The inverse of the map $\log$ is denoted by $\exp$.
These maps $\exp$ and $\log$ are ${\rm Aut}(N)$-equivariant.

The following lemma is  in \cite[Prop. 2.1]{Ra3}.

\bl
\label{lemuni}
Let $N$ be an algebraic unipotent $p$-adic Lie group. Then the map
$$
\g n\times\g n\ra \g n;
(X,Y)\mapsto \log(\exp X\exp Y)
$$
is polynomial and is  given by the Baker-Campbell-Hausdorff formula.
In other words, a
 $p$-adic Lie group $N$
is algebraic unipotent if and only if
it is isomorphic to the group of $\mathbb Q_p$-points of
a unipotent algebraic group defined over $\m Q_p$.
\el

These groups have been studied in \cite[Sect.2]{Ra3} 
(where they are called {\it quasiconnected}).

In particular, we have

\bc Let $N$ be an algebraic unipotent $p$-adic Lie group. Then the exponential and logarithm maps of $N$ are continuous.
\ec

\begin{proof}[Proof of Lemma \ref{lemuni}]
In this proof, we will say that a $p$-adic Lie group is {\it strongly
algebraic unipotent} if it is 
isomorphic to the group of $\mathbb Q_p$-points of
a unipotent algebraic group defined over $\m Q_p$.
Such a group is always algebraic unipotent.

The aim of this proof is to check the converse. 
Let $N$ be an algebraic unipotent $p$-adic Lie group.
We want to prove that $N$ is strongly algebraic unipotent. 
Its Lie algebra $\g n$ contains a flag 
$0\subset \g n_1\subset\cdots\subset\g n_r=\g n$ of ideals 
with $\dim \g n_i=i$.
We will prove, by induction on $i\geq 1$, 
that the set $N_i:=\exp(\g n_i)$ is a closed subgroup of $N$
which is strongly algebraic unipotent. 

By the induction assumption, the  set $N_{i-1}$ is
a  strongly algebraic unipotent closed subgroup. Since $\g n_{i-1}$ is an ideal, this subgroup $N_{i-1}$ is normal.
Let $X_i$ be an element of $\g n_i\smallsetminus \g n_{i-1}$ and 
$N'_i$ the  semidirect product
$N'_i:=\m Q_p\ltimes N_{i-1}$ where the action of 
$t\in \m Q_p$ by conjugation on $N_{i-1}$ is given by
$t\exp(X)t^{-1}=\exp(e^{t{\rm ad}X_i}X)$, for all $X$ in $\g n_{i-1}$. 
By construction this group $N'_i$ is strongly algebraic unipotent and
the map $\psi:N'_i\ra N;(t,n)\mapsto \exp(t X_i)n$
is a group morphism. 
Since $N'_i=\exp(\g n'_i)$, the set $N_i=\exp(\g n_i)$
is equal to the image
$\psi(N'_i)$. Hence, by Lemma \ref{lemunipro} below, the set $N_i$ 
is a closed subgroup which  is isomorphic to $N'_i$ and hence $N_i$ is strongly algebraic unipotent. 
\end{proof}

The following lemma tells us that an algebraic unipotent 
Lie subgroup of a $p$-adic Lie group is always closed.

\bl
\label{lemunipro}
Let $G$ be a totally discontinuous locally compact topological group, $N$ be an algebraic unipotent 
$p$-adic Lie group and 
$\ph:N\ra G$ be an injective morphism.
Then $\ph$ is a proper map. In particular, $\ph(N)$ is a closed subgroup of $G$
and $\ph$ is is an isomorphism of topological groups from $N$ onto 
$\ph(N)$. 
\el

\begin{proof}
If $\ph$ was not proper, there would exist a sequence $Y_n$ in 
the Lie algebra $\g n$ of $N$ such that 
$$\D\lim_{n\ra\infty}\ph(\exp(Y_n))=e
\;\;{\rm and}\;\;
\D\lim_{n\ra\infty}\|Y_n\|=\infty.
$$ 
We write $Y_n=p^{-k_n}X_n$ with integers $k_n$ going to $\infty$
and $\|X_n\| = 1$. 
Since the group $G$ admits a basis of compact open subgroups, one also has
$\D\lim_{n\ra\infty}\ph(\exp(X_n))=e$. 
Let $X$ be a cluster point of the sequence $X_n$. 
One has simultaneously, $\ph(\exp(X))=e$ and $\| X\|=1$. 
This contradicts the injectivity of $\ph$.
\end{proof}

\subsection{Lifting one-parameter morphisms}
\label{seclifmor}
\bq
We now explain how to lift one-parameter morphisms.
\eq
\bl
\label{lemlifmor}
Let $G$ be a $p$-adic Lie group, $N\subset G$ 
a normal algebraic unipotent closed subgroup, $G':=G/N$, and 
$\pi:G\ra G'$ the projection. Then, for any 
one-parameter morphism $\ph'$ of $G'$,
there exists a one-parameter morphism $\ph$ of $ G$
which lifts $\ph'$, i.e. such that $\ph'=\pi\circ\ph$.

If $\ph'$ has zero derivative, one can choose $\ph$ to have zero derivative.
\el

\begin{proof}
According to Lemma \ref{lemunipro} the image $\ph'(\m Q_p)$ is a closed subgroup
of $G'$. Hence we can assume that $\ph'(\m Q_p)=G'$.
\vs

\noi{\bf First Case}~: $N$ is central in $G$. For $k\geq 1$, we introduce 
the subgroup $Q'_k$ of $G'$ spanned by the element 
$g'_k:=\ph'(p^{-k})$. 
Since $Q'_k$ is cyclic and $N$ is central, the group  $Q_k:=\pi^{-1}(Q'_k)$
is abelian. Since the increasing union of these groups $Q_k$ is dense in $G$,
the group $G$ is also abelian.
Since $N$ is infinitely $p$-divisible, one can construct, by induction on $k\geq 0$,
a sequence $(g_k)_{k\geq 0}$ in $G$ such that 
$$
\pi(g_k)=g'_k
\;\;{\rm and}\;\;
p\, g_{k+1}=g_k\; .
$$

We claim that $g_0^{p^k}\td{k}{\infty}e$. Indeed, since $\pi(g_0)^{p^k}\td{k}{\infty}e$ and since every element $h$ of a $p$-adic Lie group that is close enough to the identity element satisfies $h^{p^k}\td{k}{\infty}e$, one can find $\ell\geq 0$ and $n$ in $N$ such that
$(g_0^{p^\ell}n^{-1})^{p^k}\td{k}{\infty}e$. As $N$ is algebraic unipotent, we have $n^{p^k}\td{k}{\infty}e$, and the claim follows, since $N$ is central.

Now, the formulae $\ph(p^{-k})=g_k$, for all $k\geq 0$, define a unique
one-parameter morphism $\ph$ of $G$ which lifts $\ph'$.
  
Note that when $\ph'$ has zero derivative, one can assume, after a repara\-metrization
of $\ph'$, 
that $\ph'(\m Z_p)=0$ and choose the sequence $g_k$ so that $g_0=e$. 
Then the morphism $\ph$  has  also zero derivative. 
\vs

\noi{\bf General Case} ~:
The composition of $\ph'$ with the action by conjugation on the 
abelianized group $N/[N,N]\simeq \m Q_p^d$ is a one-parameter morphism
$\psi:\m Q_p\ra {\rm GL}(d,\m Q_p)$. According to 
Lemma \ref{lemexpgl}, there exists  a nilpotent matrix $X$ such that 
$\psi(t)=\exp(tX)$ for all $t\in \m Q_p$. 
The image of this matrix $X$ corresponds to an algebraic unipotent subgroup $N_1$ with
$[N,N]\subset N_1\subsetneq N$ which is normal in $G$ 
and such that $N/N_1$ is a central subgroup of the group $G/N_1$. 
According to the first case, the morphism $\ph'$
can be lifted as a morphism $\ph'_1$ of $ G/N_1$.
By an induction argument on the dimension of $N$, 
this morphism $\ph'_1$ can be lifted as a morphism of $G$.
\end{proof}

Let $G$ be a  $p$-adic Lie group.
We recall the notation 
\begin{equation}
\label{eqngg}
\g g_G:=\{ X\in \g g\;\;
\mbox{\rm  derivative of a one-parameter morphism of $G$}\}.
\end{equation}

Note that this set $\g g_G$ is invariant under the adjoint action of $G$.

The following lemma tells us 
various stability properties by extension 
when the normal subgroup is algebraic unipotent.

\bl
\label{lemngq}
Let $G$ be a $p$-adic Lie group, $N$ a  normal 
algebraic unipotent subgroup of $G$, and $G':=G/N$.\\
$a)$ One has the equivalence\\
\centerline{$G$ is algebraic unipotent $\Longleftrightarrow$
$G'$ is algebraic unipotent.}
$b)$ Let $X$ in $\g g$ and $X'$ its image in $\g g'=\g g/\g n$.
One has the equivalence\\
\centerline{$X\in \g g_G \Longleftrightarrow
X'\in \g g'_{G'}$.}
$c)$ One has the equivalence\\
\centerline{$G$ is weakly regular $\Longleftrightarrow$
$G'$ is weakly regular.}
\el   

Later on in Corollary \ref{corngq} we will be able to improve this Lemma.

\begin{proof} We denote by $\pi:G\ra G/N$ the natural projection.

$a)$ The implication $\Rightarrow$ is well-known. Conversely,
we assume that $N$ and $G/N$ are algebraic unipotent  
and we want to prove that $G$ is algebraic unipotent.
Arguing by induction on $\dim G/N$, we can assume 
$\dim G/N=1$, i.e. that  there exists an isomorphism
$\ph':\m Q_p\ra G/N$. 
According to Lemma \ref{lemlifmor}, one can find a one-parameter morphism
$\ph$ of $G$ that lifts $\ph'$. By Lemma \ref{lemunipro}, the image 
$Q:=\ph(\m Q_p)$ is closed and $G$ is the semidirect product
$G=Q\ltimes N$.  By Lemma \ref{lemexpgl},  
the one-parameter morphism $t\mapsto {\rm Ad}_{\g n}\ph(t)$ is  unipotent,
and hence the group $G$ is algebraic unipotent.  

$b)$ The implication $\Rightarrow$ is easy. Conversely,
we assume that $X'$ is the derivative of a one-parameter morphism $\ph'$
of $G'$.
When $X'=0$, the element $X$ belongs to $\g n$ and, since $N$ is algebraic unipotent, 
$X$ is
the derivative of a one-parameter morphism $\ph$
of $N$. We assume now that $X'\neq 0$ so that the group 
$Q':=\ph'(\m Q_p)$ is algebraic unipotent and isomorphic to $\m Q_p$.
According to point $a)$, the group $H:=\pi^{-1}(Q')$ is algebraic unipotent.
Since the element $X$ belongs to the Lie algebra $\g h$ of $H$,
it is
the derivative of a one-parameter morphism $\ph'$
of $H$.

$c)$ $\Rightarrow$ We  assume that $G$ is weakly regular. 
Let $\ph'_1$ and $\ph'_2$ be one-parameter morphisms of $G'$
with the same derivative $X'\in \g g'$. We want to prove that $\ph'_1=\ph'_2$.
If this derivative $X'$ is zero, by Lemma \ref{lemlifmor}, 
we can lift both $\ph'_1$ and $\ph'_2$
as one-parameter morphisms of $G$ with zero derivative.
Since $G$ is weakly regular, both $\ph_1$ and $\ph_2$ are trivial and $\ph'_1=\ph'_2$.
We assume now that the derivative $X'$ is non-zero. 
As above, for $i=1$, $2$, the groups $Q'_i:=\ph'_i(\m Q_p)$ and $H_i:=\pi^{-1}(Q'_i)$
are algebraic unipotent. 
Since $G$ is weakly regular, and its algebraic unipotent subgroups 
$H_1$ and $H_2$ have the same Lie algebra, one gets successively $H_1=H_2$, 
$Q'_1=Q'_2$, and $\ph'_1=\ph'_2$.

$\Leftarrow$ We assume that $G/N$ is weakly regular. 
Let $\ph_1$ and $\ph_2$ be one-parameter morphisms of $G$
with the same derivative $X\in \g g$. We want to prove that $\ph_1=\ph_2$.
Since $G/N$ is weakly regular the one-parameter morphisms
$\pi\circ\ph_1$ and $\pi\circ \ph_2$ are equal and their image $Q'$ 
is a unipotent algebraic subgroup of $G'$. According to point $a)$,
the group $H:=\pi^{-1}(Q')$ is algebraic unipotent. 
Since $\ph_1$ and $\ph_2$ take their values in $H$,
one has $\ph_1=\ph_2$.
\end{proof}

\subsection{Unipotent subgroups tangent to a nilpotent Lie algebra}
\label{secunitan}
\bq
Proposition \ref{prongggn} below describes the nilpotent Lie subgroups 
of a weakly regular  $p$-adic Lie group $G$ which are spanned by
one-parameter morphisms.
\eq
The following proposition is due to Ratner in \cite[Thm. 2.1]{Ra3}.
\bp
\label{prongggn} 
Let $G$ be a weakly regular $p$-adic Lie group and $\g n\subset \g g$ be
a nilpotent Lie subalgebra.
Then the set $\g n_G:=\g n\cap\g g_G$ is an ideal of $\g n$
and there exists an algebraic unipotent subgroup $N_G$ of $G$ 
with Lie algebra $\g n_G$. 
\ep 

This group $N_G$ is unique. It is a closed subgroup of $G$. By construction,
it is the largest algebraic unipotent subgroup whose Lie algebra 
is included in $\g n$. 

\begin{proof}
We argue by induction on $\dim \g n$. We can assume $\g n_G\neq 0$.
\vs

{\bf First case}~: $\g n$ is abelian. 
Let $X_1,\ldots ,X_r$ be a maximal family of linearly indepent elements
of $\g n_G$ and, for $i \leq r$, let  $\ph_i$ be the one-parameter
morphism with derivative $X_i$. Since $G$ is weakly regular, the
group spanned by the images  $\ph_i(\m Q_p)$ is commutative and 
the map
$$
\ps:\m Q_p^d\ra G; (t_1,\ldots,t_r)\mapsto \ph_1(t_1)\ldots\ph_r(t_r)
$$
is an injective morphism. Its image is a unipotent algebraic 
subgroup $N_G$ of $G$ whose Lie algebra is $\g n_G$.
\vs

{\bf Second case}~: $\g n$ is not abelian. 
Let $\g z$ be the center of $\g n$ and $\g z_2$ the ideal of $\g n$
such that $\g z_2/\g z$ is the center of $\g n/\g z$.
If $\g n_G$ is included in the centralizer $\g n'$ of $\g z_2$,
we can  apply the induction hypothesis to $\g n'$. 
We assume now that $\g n_G$ is not included in $\g n'$, i.e.
there exists 
$$
X\in \g n_G
\;{\rm and}\;
Y\in \g z_2\;
\mbox{\rm such that}\;
[X,Y]\neq 0.
$$ 
This element $Z:= [X,Y]$ belongs to the center $\g z$. 

We first 
check that $Z$ belongs also to $\g g_G$. 
Indeed, let $\g m$ be the $2$-dimensional Lie subalgebra of $\g n$
with basis $X,Z$. This Lie algebra is normalized by $Y$. 
For $\eps\in \m Q_p$ small enough, there exists a group morphism 
$\psi: \eps \m Z_p\ra G$ whose derivative at $0$ is $Y$, and one has 
${\rm Ad}(\psi(\eps))Y=e^{\eps\,{\rm ad}Y}X=X\!-\!\eps Z$.
Since $X$
belongs to $\g g_G$,  
the element $X\!-\!\eps Z$ also belongs to $\g g_G$.
By the first case applied to the abelian Lie subalgebra $\g m$, the 
element $Z$ belongs to $\g g_G$. 

This means that there exists a one-parameter subgroup $U$ of $G$ 
whose Lie algebra is $\g u=\m Q_pZ$. Let $C$ be the centralizer 
of $U$ in $G$. According to Lemma \ref{lemngq}, 
the quotient group $C/U$ is also weakly regular.
We apply our induction hypothesis to this group $C':=C/U$ and the 
nilpotent Lie algebra $\g n/\g u$. There exists a largest
algebraic unipotent subgroup $N'_{C'}$ in $C'$ whose Lie algebra 
is included in $\g n'$. Hence, using again Lemma \ref{lemngq}, 
there exists a largest
algebraic unipotent subgroup $N_{C}$ of $C$ whose Lie algebra 
is included in $\g n$. Since $G$ is weakly regular, any one-parameter
subgroup of $G$ tangent to $\g n$ is included in $C$ and 
$N_C$ is also the largest
algebraic unipotent subgroup of $G$ whose Lie algebra 
is included in $\g n$.
\end{proof}

\subsection{Largest normal algebraic unipotent subgroup}
\label{seclarnoruni}
\bq
We prove in this section that a weakly regular $p$-adic Lie group
contains a largest normal algebraic unipotent subgroup.
\eq
 
Let $G$ be a $p$-adic Lie group.
We denote by $\ol G_u$  the closure of the subgroup $G_u$ of $G$
generated by all the one-parameter 
subgroups of $G$. 
This group $\ol G_u$ is normal in $G$.
We denote by $\g g_u$ the Lie algebra of $\ol G_u$. 
It is an ideal of $\g g$.

We recall that the radical $\g r$ of $\g g$ 
is the largest solvable ideal of $\g g$ and that the 
nilradical $\g n$ of $\g g$ is the largest nilpotent ideal of $\g g$. The nilradical is the set of $X$ in $\g r$ such that ${\rm ad}X$ is nilpotent and one has 
$[\g g,\g r]\subset\g n$ (see \cite{Bou1}).

When $G$ is weakly regular, we denote by $R_u$ the largest algebraic unipotent subgroup
of $G$ whose Lie algebra is included in $\g n$. It exists
by Proposition \ref{prongggn}. 

The following proposition is mainly in \cite[Lem. 2.2]{Ra3}.

\bp
\label{proruru}
Let $G$ be a weakly regular $p$-adic Lie group.\\
$a)$ The group $R_u$ is the largest normal 
algebraic unipotent subgroup of $G$.\\
$b)$ Its Lie algebra $\g r_u$ is equal to 
$\g r_u=\g n\cap \g g_G=\g r\cap\g g_G$.\\
$c)$ One has the inclusion $[\g g_u,\g r]\subset \g r_u$.\\
$d)$ Let $G'=G/R_u$. 
Let $X$ be in $\g g$ and $X'$ be its image in $\g g'=\g g/\g r_u$.
One has the equivalence\\
\centerline{$X\in \g g_G \Longleftrightarrow
X'\in \g g'_{G'}$.}
\ep

\begin{proof}
a) We have to prove that any normal algebraic unipotent subgroup $U$
of $G$ is included in $R_u$. Indeed the Lie algebra $\g u$ of $U$ 
is a nilpotent ideal of $\g g$, hence it is included in $\g n$ and $U$
is included in $R_u$.

$b)$ We already know the equality $\g r_u=\g n\cap \g g_G$ from 
Proposition \ref{prongggn}. It remains to check the inclusion 
$\g r\cap \g g_G\subset \g n$. Indeed, let $X$ be an element in $\g r\cap\g g_G$.  
Since $X$ is the derivative of a one-parameter morphism, 
by Lemma \ref{lemexpgl}, the endomorphism ${\rm ad}X$ is nilpotent. 
Since $X$ is also in the radical $\g r$, $X$ has to be in the nilradical $\g n$. 

$c)$ We want to prove that the adjoint action of $G_u$ on the quotient 
Lie algebra $\g r/\g r_u$ is trivial. That is, we want to prove
that, for all
$$
X\in \g g_G
\;\;{\rm and}\;\;
Y\in \g r
\; ,\;\mbox{\rm one has}\;\;
[X,Y]\in \g r_u.
$$
By Lemma \ref{lemexpgl}, the endomorphism ${\rm ad}X$ is nilpotent.
Since $Y$ is in $\g r$, the bracket $[X,Y]$ belongs to $\g n$
and the vector space $\g m:=\m Q_p X\oplus \g n$ is a nilpotent Lie algebra
normalized by $Y$. Hence, by Proposition \ref{prongggn}, 
the set $\g m\cap \g g_G$ is a nilpotent Lie algebra. 
This set is normalized by $Y$ since it is invariant by $e^{\eps\,{\rm ad}Y}$
for $\eps\in \m Q_p$ small enough. In particular the element $[X,Y]$ belongs to $\g g_G$ and hence to $\g r_u$.

$d)$ This is a special case of Proposition \ref{prongggn}. 
\end{proof}

\section{Derivatives of one-parameter morphisms}
\label{secderone}
\bq
In this chapter, we describe the set $\g g_G$ of derivatives of one-parameter
morphisms of a weakly regular $p$-adic Lie group $G$.
\eq

\subsection{Construction of one-parameter subgroups}
\label{seccononepar}

\bq
We explain first a construction of one-parameter morphisms of $G$
borrowed from \cite{BQII} and \cite{BQLS}.
\eq

\begin{Lem}
\label{lemconstrunipo} 
Let $G$ be a $p$-adic Lie group and $g\in G$.
Then the vector space 
$\g g_g^+:=\{ v\in \g g\mid {\D\lim_{n\ra\infty}}{\rm Ad}g^{-n}v=0\}$
is included in $\g g_G$.
\end{Lem}

Note that $\g g_g^+$ is a nilpotent Lie subalgebra of $\g g$.

The proof relies on the existence of compact open subgroups of $G$
for which the exponential map satisfies a nice equivariant property.
We need some classical definition (see \cite{DSMS}).
A $p$-adic Lie group $\Om$ is said
to be a {\it standard} group if there exists a $\m Q_p$-Lie algebra $\g l$
and a compact open sub-$\m Z_p$-algebra $O$ of $\g l$ such that
the Baker-Campbell-Hausdorff series converges 
on $O$ and $\Om$ is isomorphic
to the $p$-adic Lie group $O$ equipped with 
the group law defined by this formula.
In this case, $\g l$ identifies canonically with the Lie algebra of $\Om$, 
every element of $\Om$ admits a logarithm
and the logarithm map induces an isomorphism $\Om\ra O$. 
If $G$ is any $p$-adic Lie group,
it admits a standard open subgroup (see \cite[Theorem 8.29]{DSMS}).
If $\Omega$ is such a subgroup and if $O$ is the associated
compact open sub-$\m Z_p$-algebra  of $\g g$, we denote by
$\exp_\Omega:O\ra\Omega$ the inverse diffeomorphism 
of the logarithm map $\Omega\ra O$.

Note that if $\Omega$ and $\Omega'$ are standard open subgroups of $G$,
the maps $\exp_\Omega$ and $\exp_{\Omega'}$ 
coincide in some neighborhood of $0$ in $\g g$.

\bl
\label{lemexpmap}
Let $G$ be a $p$-adic Lie group, $\Om\subset G$ a standard open subgroup
and ${\rm exp}_\Om:O\ra \Om$ the corresponding exponential map.
For every compact subset $K\subset G$,
there exists an open subset $O_K\subset \g g$ which is contained in $O$ and in all
the translates ${\rm Ad}g^{-1}(O)$, $g\in K$, and
such that one has the equivariance property
\begin{equation*}\label{eqnexpmap}
{\rm exp}_\Om({\rm Ad}g(v))= g\, {\rm exp}_\Om(v)\, g^{-1}
\;\;\;
\text{for any $v\in O_K$, $g\in K$}.
\end{equation*}
\el

\begin{proof} We may assume that $K$ contains $e$.
The intersection $\Om_K:=\cap_{g\in K}g^{-1}\Om g$ 
is an open neighborhood of $e$ in $G$.
We just choose $O_K$ to be the open set $O_K:= \log(\Om_K)$.
\end{proof}

\begin{proof}[Proof of Lemma \ref{lemconstrunipo}] 
This is \cite[Lem. 5.4]{BQII}. For the sake of completeness, we recall the proof.
Fix $g\in G$. Let $\Omega$ be a standard open subgroup of $G$
with exponential map $\exp_\Omega:O\ra\Omega$.
By Lemma \ref{lemexpmap}, there exists an open additive subgroup
$U\subset O\cap \g g_g^+$ such that ${\rm Ad}g^{-1}U\subset O$ and that 
$$
\exp_\Omega(u)=g\exp_\Omega({\rm Ad}g^{-1}u)g^{-1}
\;\;\mbox{\rm for any $u$ in $U$}.
$$
After eventually replacing $U$ by $\bigcap_{k\geq 0}{\rm Ad}g^{k}U$, 
we can assume ${\rm Ad}g^{-1}U\subset U$.
Now, for $k\geq 0$, let $U_k:={\rm Ad}g^{k} U$
and define a continuous map $\ps_k:U_k\ra G$
by setting 
$$
\ps_k(u)=g^{k}\exp_{\Omega}({\rm Ad}g^{-k}u)g^{-k}
\;\;\mbox{\rm for any $u$ in $U_k$}.
$$ 
We claim that, for any $k$,
one has $\ps_k=\ps_{k-1}$ on $U_{k-1}={\rm Ad}g^{-1}U_k$. 
Indeed, let $u$ be in $U_{k}$.
As $u_k:= {\rm Ad}g^{-k}u$ belongs to $U$, we have
$$
\ps_{k+1}(u)=g^{k}(g\exp_\Omega({\rm Ad}g^{-1}u_k)g^{-1})g^{-k}=
g^{k}\exp_\Omega(u_k)g^{-k}=\ps_k(u).$$
Therefore, as $\g g_g^+=\bigcup_{k\geq 0}U_k$,
one gets a map $\ps:\g g_g^+\ra G'$ whose restriction to any $U_k$,
$k\geq 0$, is $\ps_k$.
For every $v$ in $\g g_g^+$, the map $t\mapsto \ps(tv)$ is a one-parameter
morphism of $G$ whose derivative is equal to $v$.
\end{proof}

\subsection{The group $G_{nc}$ and its Lie algebra $\g g_{nc}$}
\label{secgrognc}
\bq
We introduce in this section a normal subgroup $G_{nc}$
of $G$ which  contains $G_u$.
\eq

Let $G$ be a $p$-adic Lie group and 
$\g g_{nc}$ be the smallest ideal of the Lie algebra $\g g$ 
such that the Lie algebra $\g s_c:=\g g/\g g_{nc}$ is semisimple
and such that the adjoint group 
${\rm Ad}_{\g s_c}(G)$ is bounded in the group 
${\rm Aut}(\g s_c)$ of automorphisms of $\g s_c$.
Let $G_{nc}$ be the kernel of the adjoint action in $\g s_c$, i.e.
$$
G_{nc}:=\{ g\in G\mid 
\mbox{\rm for all $X$ in $\g g$}\, ,\;\;
{\rm Ad}g(X)-X\in \g g_{nc}\}
$$
By construction $G_{nc}$ is a closed normal subgroup of $G$ with Lie algebra $\g g_{nc}$.

\bl
\label{lemgnc}
Let $G$ be a $p$-adic Lie group. Any one-parameter morphism
of $G$ takes its values in $G_{nc}$
\el

In other words,
the group $G_u$ is included in $G_{nc}$.

\begin{proof}
Let $\ph$ be a one-parameter morphism of $G$. 
Then ${\rm Ad}_{\g s_c}\circ\ph$ is a one-parameter morphism of 
${\rm Aut}(\g s_c)$ whose image is relatively compact. 
By Lemma \ref{lemexpgl}, this one-parameter morphism is trivial 
and $\ph$ takes its values in $G_{nc}$.  
\end{proof}

\subsection{Derivatives and Levi subalgebras}
\label{secderlevsub}
\bq
We can now describe precisely which elements of $\g g$ 
are tangent to one-parameter subgroups of $G$.
\eq

We recall that an element $X$ of a semisimple Lie algebra $\g s$ 
is said to be nilpotent if the endomorphism ${\rm ad}_{\g s}X$ 
is nilpotent. In this case, for any finite dimensional 
representation $\rho$ of $\g s$, 
the endomorphism $\rho(X)$ is also nilpotent.

We recall that a Levi subalgebra $\g s$ of a Lie algebra $\g g$
is a maximal semisimple Lie subalgebra, 
and that one has the Levi decomposition  $\g g=\g s\oplus\g r$.

The following proposition is proven in \cite[Th. 2.2]{Ra3}
under the additional assumption that $G$ is ${\rm Ad}$-regular.

\bp
\label{proxxsxr}
Let $G$ be a weakly regular $p$-adic Lie group,
$\g r$ be the radical of $\g g$, $\g s$ a Levi subalgebra of $\g g$
and $\g s_u:=\g s\cap \g g_{nc}$. One has the equality~:
$\g g_G=\{ X\in \g s_u\oplus\g r_u \mid {\rm ad} X$ is nilpotent$\}.$
\ep

It will follow from Lemma \ref{lemxxsxr} that the Lie algebra 
$\g s_u\oplus\g r_u$ does not depend on the choice of $\g s$.

The key ingredient in the proof of Proposition \ref{proxxsxr} 
will be Lemma \ref{lemconstrunipo}.
We will begin by three preliminary lemmas. The first two lemmas are classical.

\bl
\label{lemeigval}
Let $V=\m Q_p^d$ and $G$ be a subgroup of ${\rm GL}(V)$ such that 
$V$ is an unbounded and irreducible reprentation of $G$. 
Then $G$ contains an element $g$ with at least one eigenvalue 
of modulus not one.
\el

\begin{proof}
Let $A$ be the associative subalgebra of ${\rm End}(V)$ spanned by $G$.
Since $V$ is irreducible, the associative algebra $A$ is semisimple 
and the bilinear form $(a,b)\mapsto tr(ab)$ is non-degenerate on $A$
(see \cite[Ch. 17]{La65}). 
If all the eigenvalues of all the elements of $G$ have modulus $1$,
this bilinear form is bounded on $G\times G$. 
Since $A$ admits a basis included in $G$,
for any $a$ in $A$, the linear forms $b\mapsto tr(ab)$ on $A$ are bounded on 
the subset $G$. Hence  $G$ is a bounded
subset of $A$. 
\end{proof}

\bl
\label{lemauts0}
Let $\g s_0$ be a simple Lie algebra over $\m Q_p$ and $S_0:={\rm Aut}(\g s_0)$.
All open unbounded subgroups $J$ of $S_0$ have finite index in $S_0$.
\el


\begin{proof}
Since $J$ is unbounded, by Lemma \ref{lemeigval}, it contains an element $g_0$ 
with at least one eigenvalue 
of modulus not one.
Since $J$ is open, the unipotent Lie subgroups
\begin{eqnarray*}
U^+\!=\!\{ g\in S_0\mid \lim\limits_{n\ra\infty}g_0^{-n}gg_0^n=e\}
\;{\rm and}\;
U^-\!=\!\{ g\in S_0\mid \lim\limits_{n\ra\infty}g_0^{n}gg_0^{-n}=e\}
\end{eqnarray*}
are included in $J$.
By \cite[6.2.$v$ and 6.13]{BoTi73}, $J$ 
has finite index in $S_0$.
\end{proof}

The third lemma  contains the key ingredient.

\bl
\label{lemxxsxr}
Let $G$ be a weakly regular $p$-adic Lie group,
$\g r$ the radical of $\g g$, $\g s$ a Levi subalgebra of $\g g$
and $\g s_u:=\g s\cap \g g_{nc}$.\\
$a)$ One has the inclusion $[\g s_u,\g r]\subset \g r_u$.\\ 
$b)$ Every nilpotent element $X$ in $\g s_u$ belongs to $\g g_G$.
\el

Note that Lemma \ref{lemxxsxr}.$a$ does not follow 
from Proposition \ref{proruru}.$c$, since with the definitions of $\g g_u$
and $\g s_u$ that we have given, we do not know yet that $\g s_u=\g s\cap \g g_u$.

\begin{proof}
a) By Proposition \ref{proruru}, we can assume $\g r_u=0$.
We want to prove that $[\g s_u,\g r]=0$. 
Let $\g s_i$, $i=1,\ldots,\ell$ be the simple ideals of $\g s_u$.
Replacing $G$ by a finite index subgroup, 
we can also assume that the ideals $\g s_i\oplus \g r$ are $G$-invariant. 
Similarly, let $\g r_j$, $j=1,\ldots, m$ be the simple subquotients
of a Jordan-H\"{o}lder sequence of the $G$-module $\g r$.

On the one hand, by assumption, 
for all $i\leq \ell$, the group 
${\rm Ad}_{\g s_i\oplus\g r/\g r}(G)$ is unbounded.
Hence, by Lemma \ref{lemeigval}, there exists an element $g_i$ in $G$
and $X_i$ in $\g g_{g_i}^+\cap(\g s_i\oplus\g r)$ 
whose image in $\g g/\g r$ is  non zero. 
By Lemma \ref{lemconstrunipo}, there exists a 
one-parameter morphism $\ph_i$ of $G$ whose derivative is $X_i$.  

On the other hand, since $\g r_u= 0$, 
by the same Lemma \ref{lemconstrunipo}, 
for every $g$ in $G$, all the eigenvalues of ${\rm Ad}_{\g r}(g)$
have modulus $1$.
By Lemma \ref{lemeigval}, for all $j\leq m$, 
the image ${\rm Ad}_{\g r_j}(G)$ 
of $G$ in any simple 
subquotient $\g r_j$ is bounded. 
In particular, the one-parameter morphisms
${\rm Ad}_{\g r_j}\circ\ph_i$ are bounded. 
Hence, by Lemma \ref{lemexpgl}, one has ${\rm ad}_{\g r_j}(X_i)=0$.
Since
$\g r_j$ is a simple $\g g$-module, the Lie algebra 
${\rm ad}_{\g r_j}(\g g)$ is reductive and 
contains ${\rm ad}_{\g r_j}(\g s_i)$
as an ideal.
Since $\g s_i$ is a simple Lie algebra, this implies  ${\rm ad}_{\g r_j}(\g s_i)=0$.
Since the action of $\g s_u$ on $\g r$ is semisimple, this implies
the equality $[\g s_u,\g r]= 0$.

$b)$ 
As in $a)$, we can assume that $G$
preserves the  ideals $\g s_i\oplus \g r$ and that $\g r_u= 0$.
According to this point $a)$, the Lie algebras $\g s_u$ and $\g r$ 
commute and hence $\g s_u$ is the unique Levi subalgebra of $\g g_{nc}$
(see \cite[\S 6]{Bou1}).
In particular, 
\begin{equation}
\label{eqnadgss}
\mbox{\rm for all $g$ in $G$, one has ${\rm Ad}g(\g s_u)=\g s_u$}.
\end{equation}
Let $X$ be a nilpotent element of $\g s_u$. We want to prove that
$X$ is the derivative of a one parameter morphism of $G$.
By Jacobson-Morozov theorem, there exists an automorphism $\psi$ of 
$\g s_u$ such that $\psi(X)= p^{-1} X$. 
Since, for every simple ideal $\g s_i$ of $\g s$, 
the subgroup ${\rm Ad}_{\g s_i\oplus\g r/\g r}(G_{nc})\subset {\rm Aut}(\g s_i)$ 
is unbounded and open, 
this subgroup has finite index. 
Hence, remembering also \eqref{eqnadgss}, 
there exists $k\geq 1$ and $g$ in $G_{nc}$ such that 
${\rm Ad} g(X)=p^{-k}X$.
Then, by Lemma \ref{lemconstrunipo}, $X$ is the derivative
of a one-parameter morphism of $G$.  
\end{proof}

\begin{proof}[Proof of Proposition \ref{proxxsxr}] 
We just have to gather what we have proved so far. 
By Proposition \ref{proruru}, we can assume $\g r_u=0$.
Let $X\in \g g$. We write
$X=X_{\g s}+X_{\g r}$ with
$X_{\g s}\in \g s$ and $X_{\g r}\in \g r$.

Proof of the inclusion $\subset$. Assume that $X$ is the derivative of a one-parameter
morphism $\ph$ of $G$. 
By Lemma \ref{lemgnc}, 
$X$ belongs to $\g g_{nc}$
and hence $X_{\g s}$ belongs to $\g s_u$. 
By Lemma \ref{lemexpgl}, the endomorphism ${\rm ad}_{\g g}X$ is nilpotent,
and hence $X_{\g s}$ is a nilpotent element of the semisimple Lie algebra $\g s$. 
According to Lemma \ref{lemxxsxr}, 
$X_{\g s}$ and $X_{\g r}$ commute and $X_{\g s}$ is the derivative of a 
one-parameter morphism $\ph_{\g s}$ of $G$. Then 
$X_{\g r}$ is also the derivative of a 
one-parameter morphism $\ph_{\g r}$ of $G$, the one given by 
$t\mapsto \ph_{\g s}(t)^{-1}\ph(t)$. 
Hence $X_{\g r}$ belongs to $\g r_u$.

Proof of the inclusion $\supset$.
Assume that $X_\g s$ belongs to $\g s_u$,  $X_\g r$ belongs to $\g r_u$
and ${\rm ad}X$ is nilpotent.
By Lemma \ref{lemxxsxr},
$X_{\g s}$ and $X_{\g r}$ commute and $X_{\g s}$ is the derivative of a 
one-parameter morphism $\ph_{\g s}$ of $G$. By assumption 
$X_{\g r}$ is the derivative of a 
one-parameter morphism $\ph_{\g r}$ of $G$. 
Hence $X$  is also the derivative of a 
one-parameter morphism $\ph$ of $G$,
the one given by 
$t\mapsto \ph_{\g s}(t)\ph_{\g r}(t)$. 
Hence $X$ belongs to $\g g_G$.
\end{proof}

\section{Groups spanned by unipotent subgroups}
\label{secgrospauni}

\bq
In this chapter, we prove the two main results 
Proposition \ref{prolevdec} and Theorem \ref{thwearegreg}
that we announced in the introduction.
\eq

\subsection{Semisimple regular $p$-adic Lie groups}
\label{secsemreg}
\bq
We recall first a nice result due to Prasad-Raghunathan 
which is an output from the theory of congruence subgroups
\eq

Let $\g s$ be a semisimple Lie algebra over $\m Q_p$,
${\rm Aut}(\g s)$ the group of automorphisms of $\g s$
and $S_+:={\rm Aut}(\g s)_u\subset {\rm Aut}(\g s)$ the subgroup spanned 
by the one-parameter subgroups of ${\rm Aut}(\g s)$. 
We will say that $\g s$ is {\it totally isotropic}
if $\g s$ is spanned by nilpotent elements.
In this case $S_+$ is an open finite index  subgroup of ${\rm Aut}(\g s)$,
see \cite[6.14]{BoTi73}.
Since $\g s=[\g s,\g s]$, the group  $S_+$ is perfect i.e. $S_+=[S_+,S_+]$.
In particular this group admits a universal central topological extension
$\wt S_+$, i.e. a group which is universal among 
the central topological extension of $S_+$.
In \cite{PR84} Prasad-Raghunathan were able to describe this group
$\wt S_+$
(special cases were obtained before by Moore, Matsumoto and Deodhar).
We will only need here the fact that this group is a finite extension
of $S_+$.

\bp 
\label{proprarag} {\bf (Prasad--Raghunathan) }
Let $\g s$ be a totally isotropic semisimple $p$-adic Lie algebra.
Then the group $S_+$ admits a universal topological central extension 
$$
1\longrightarrow Z_0 \longrightarrow {\wt S}_+\stackrel{\pi_0}{ \longrightarrow}
S_+ \longrightarrow 1 
$$
and its center $Z_0$ is a finite group.
\ep 
The word universal means that for all topological central extension
$$
1\longrightarrow Z_E \longrightarrow E\stackrel{\pi}{ \longrightarrow} 
S_+ \longrightarrow 1 
$$
where $E$ is a locally compact group and $Z_E$ is a closed central subgroup,
there exists a unique continuous morphism $\psi:{\wt S}_+\ra E$ 
such that $\pi_0=\pi\circ \psi$.

\begin{proof}
See \cite[Theorem 10.4]{PR84}.
\end{proof}

\begin{Rmq}
This result does not hold for real Lie groups: indeed, the center $Z_0$ of the universal cover
of ${\rm SL}(2,\m R)$ is isomorphic to $\m Z$.
\end{Rmq}

\bc
\label{corprarag} 
Let $\g s$ be a totally isotropic semisimple $p$-adic Lie algebra.
For every topological central extension
$$
1\longrightarrow Z_E \longrightarrow E\stackrel{\pi}{ \longrightarrow} 
S_+ \longrightarrow 1 
$$
with $E=\ol{[E,E]}$, the group $Z_E$ is finite.
\ec 

\begin{proof}[Proof of Corollary \ref{corprarag}.] Let $\psi:{\wt S}_+\ra E$
be the morphism given by the universal property.
Since, by Proposition \ref{proprarag}, the group $Z_0$ is finite, the projection
$\pi_0=\pi\circ \psi$ is a proper map, hence $\psi$ is also a proper map and the image $\psi({\wt S}_+)$ is a closed subgroup of $E$.
Since  $E=\psi({\wt S}_+)Z_E$,
one has the inclusion $[E,E]\subset \psi({\wt S}_+)$, and
the assumption $E=\ol{[E,E]}$ implies that the morphism $\psi$
is onto. Hence the group $Z_E=\psi(Z_0)$ is finite.
\end{proof}

\begin{Rmq}
For real Lie groups, the center $Z_E$ might even be non discrete. Such an example is given by 
the quotient $E$ of the product  $\m R\times\wt{{\rm SL}(2,\m R)}$ by a 
discrete subgroup of $\m R\times Z_0$ whose projection on $\m R$ is dense.
\end{Rmq}

\subsection{The Levi decomposition of $G_u$}
\label{seclevdec}
\bq
We prove in this section that in a weakly regular 
$p$-adic Lie group $G$, 
the subgroup $G_u$ is closed and admits a Levi decomposition.
\eq

Let $G$ be a weakly regular $p$-adic Lie group and 
$\g r$ be the solvable radical of $\g g$.
We recall that $\g g_{nc}$ is the smallest ideal of $\g g$ containing $\g r$
such that the group ${\rm Ad}_{\g g/\g g_{nc}}(G)$ is bounded.
Let $\g s$ be a Levi subalgebra of $\g g$ and $\g s_u:=\g s\cap \g g_{nc}$.

We recall that $G_u$ is the subgroup of $G$ spanned by all the one-parameter
subgroups of $G$, that $R_u$ is the subgroup of $G$ spanned by all the one-parameter
subgroups of $G$ tangent to $\g r$, and we define $S_u$ as 
the subgroup of $G$ spanned by all the one-parameter
subgroups of $G$ tangent to $\g s$.
Note that we don't know yet, but we will see it in the next proposition, 
that $S_u$ is indeed a closed subgroup
with Lie algebra equal to $\g s_u$.

\bp
\label{prolevdec}
Let $G$ be a weakly regular $p$-adic Lie group.\\
$a)$ The group $R_u$ is closed. It is the largest normal algebraic unipotent subgroup of $G$.\\
$b)$ The group $S_u$ is closed. Its  Lie algebra is $\g s_u$, and the morphism\\
${\rm Ad}_{\rm \g s_u}: S_u\ra ({\rm Aut}\,{\g s_u})_u$ is onto
and has finite kernel.\\
$c)$ The group $G_u$ is closed. One has $G_u=S_uR_u$ and $S_u\cap R_u=\{ e\}$.
\ep

\begin{Rmq} In a real Lie group, the group tangent to a Levi subalgebra is not necessary closed, as for example, if 
$G=(S\times \m T)/Z$ where $S$ is the universal cover of ${\rm SL}(2,\m R)$, $\m T=\m R/\m Z$ 
and $Z$ is the cyclic subgroup spanned by $(z_0,\al_0)$ 
with $z_0$ a generator of the center of $S$ and $\al_0$ an irrational element
of $\m T$. 
\end{Rmq}

\begin{proof}[Proof of Proposition \ref{prolevdec}]
$a)$ This follows from Proposition \ref{proruru}.

$b)$ Let $E:=\ol S_u$ be the closure 
of $S_u$ and $S_+:=({\rm Aut}\,{\g s_u})_u$. 
Note that this group $E$ normalizes $\g s_u$.
We want to apply Corollary \ref{corprarag} to the morphism 
$$
E\stackrel{\pi}{\longrightarrow} S_+ 
$$
where $\pi$ is the adjoint action $\pi:={\rm Ad}_{\rm \g s_u}$.

We first check that the assumptions of 
Corollary \ref{corprarag} are satisfied.
Since $E$ is weakly regular the kernel $Z_E$ 
of this morphism $\pi$ commutes with all the one-parameter 
subgroups tangent to $\g s_u$.
Hence $Z_E$ is equal to the center of $E$. 
Since $S_+$ is spanned by one-parameter subgroups,
and since by Proposition \ref{proxxsxr} any nilpotent element $X$ of $\g s_u$
is tangent to a one parameter subgroup $\ph$ of $E$, 
this morphism $\pi$ is surjective. 
Now,  by Jacobson Morozov Theorem, for any nilpotent element $X$ of $\g s_u$,
there exists an element 
$H$
in $\g s_u$ such that  $[H,X]=X$.
Let $\varphi:\mathbb Q_p\rightarrow G$ be the one-parameter subgroup tangent to $X$, which exists by Proposition \ref{proxxsxr}. 
Since $G$ is weakly regular, 
one has, for $t$ in $\m Q_p$ and 
$g_\eps:=\exp(\eps H)$ with $\eps$ small,
$$
g_\eps\ph(t)g_\eps^{-1}\ph(t)^{-1}=
\ph(e^{\eps}t)\ph(-t)=\ph((e^\eps-1)t).
$$
This proves that $\ph(\m Q_p)$  is included in 
the derived subgroup $[E,E]$. In particular one has $E=\ol{[E,E]}$.

According to Corollary \ref{corprarag}, the kernel $Z_E$ is finite.
In particular, one has $\dim E=\dim \g s_u$.
Since $\mathfrak s_u$ is totally isotropic, 
one can find a basis of $\g s_u$ all of whose elements are nilpotent. 
By Lemma \ref{lemxxsxr}, all these elements are in $\g g_G$. Hence,
by the implicit function theorem, the group $S_u$ is open in $E$.
Therefore, $S_u$ is also closed and $S_u=E$.

$c)$ Since the adjoint action of $R_u$ on $ \g g_{nc}/\g r$
is trivial, the intersection  
$S_u\cap R_u$ is included in the kernel $Z_E$
of the adjoint map ${\rm Ad}_{\g s_u}$. 
Since, by $b)$, this kernel is finite,
and since the algebraic unipotent group $R_u$
does not contain finite subgroups, one gets $S_u\cap R_u=\{ e\}$.

It remains to check that $S_uR_u$ is closed and that $G_u=S_uR_u$.
Thanks to Propositions \ref{proruru} and \ref{proxxsxr}, 
we can assume that $R_u=\{ e\}$. 
In this case, we know from point $b)$ that $S_u$ is closed 
and from Proposition \ref{proxxsxr} that $G_u=S_u$.
\end{proof}

Here are a few corollaries.
The first corollary is an improvement of Lemma \ref{lemngq}.

\bc
\label{corngq}
Let $G$ be a $p$-adic Lie group, $H$ a  normal 
weakly regular closed subgroup of $G$ such that $H=H_u$, and $G':=G/H$.\\
$a)$ One has the equality $G'_u=G_u/H$.\\
$b)$ Let $X$ in $\g g$ and $X'$ its image in $\g g'=\g g/\g h$.
One has the equivalence\\
\centerline{$X\in \g g_G \Longleftrightarrow
X'\in \g g'_{G'}$.}
$c)$ One has the equivalence\\
\centerline{$G$ is weakly regular $\Longleftrightarrow$
$G'$ is weakly regular.}
\ec 

\begin{Rmq}
The assumption that $H=H_u$ is important. For instance the 
group $G=\m Q_p$ and its normal subgroup $H=\m Z_p$ are 
weakly regular while the quotient $G/H$ is not weakly regular.
\end{Rmq}

\begin{proof}
We prove these three statements simultaneously.
Since $H=H_u$, according to Proposition \ref{prolevdec}, the group 
$H$ admits a Levi decomposition $H=SR$ 
where $R$ is a normal algebraic unipotent Lie subgroup 
and where $S$ is a Lie subgroup with finite center $Z$
whose Lie algebra is semisimple, totally isotropic, and
such that the adjoint map ${\rm Ad}_{\g s}:S\ra{\rm Aut}(\g s)_u$ 
is surjective.  
Note that $R$ is also a normal subgroup of $G$. 
Using Lemma \ref{lemngq}, we can assume that $R=\{ e\}$.

Let $C$ be the centralizer of $H=S$ in $G$. 
Since $H$ is normal in $G$, since $H=H_u$ and $H$ is weakly regular, $C$ is also 
the kernel of the adjoint action of $G$ on $\g h=\g s_u$.
Therefore, by Proposition \ref{prolevdec}, the image of the group morphism 
$$
H\times C\ra G;(h,c)\mapsto hc
$$ 
has finite index in $G$.
Its kernel is isomorphic to $H\cap C=Z$ and hence is finite.
When this morphism $H\times C\ra G$ is an isomorphism,
our three statements are clear.
The general case reduces to this one thanks to Lemma \ref{lemfinitelift} below.
\end{proof}

\begin{Lem} \label{lemfinitelift}
Let $G$ be a locally compact topological group, 
$Z$ be a finite central subgroup of $G$ 
and $\varphi:\mathbb Q_p\rightarrow G/Z$ 
be a continuous morphism.
Then $\varphi$ may be lifted as a continuous 
morphism $\widetilde{\varphi}:\mathbb Q_p\rightarrow G$.
\end{Lem}

\begin{proof} Let $H$ be the inverse image of $\varphi(\mathbb Z_p)$ in $G$. 
Then $H$ is totally discontinuous. 
In particular it contains an open compact subgroup $U$ 
such that $U\cap Z=\{e\}$, so that $U$ maps injectively in $G/Z$. 
Let $\ell$ be an integer such that $\varphi(p^\ell\mathbb Z_p)\subset UZ/Z$. 
After rescaling, we can assume that $\ell=0$.
We let $g_0$ be the unique element of $U$ such that $\varphi(1)=g_0 Z$. 
Since $U$ map injectively in $G/Z$, we have $g_0^{p^k}\td{k}{\infty}e$. 

Let $X$ be the group of elements of $p$-torsion in $Z$ and 
$Y$ be the group of elements whose torsion is prime to $p$. 
For any $k\geq 0$ pick some 
$g_k$ in $G$ such that $\varphi(p^{\ell -k})=g_kZ$ and 
let $x_k$ and $y_k$ be the elements of $X$ and $Y$ such that 
$g_k^{p^k}=g_0 x_ky_k$. 
We let $z_k$ be the unique element of $Y_k$ such that $z_k^{p^k}=y_k$. 
Replacing $g_k$ by $g_k z_k^{-1}$, we can assume that $z_k=e$.
Since $x_k$ only takes finitely many values, we can find a $x$ in $X$ and 
an increasing sequence $(k_n)$ such that, for any $n$, 
$x_{k_n}=x$. Now, since $x$ is a central $p$-torsion element, 
one has $(g_0 x)^{p^k}\td{k}{\infty}e$. Since, for any $n$,
$g_{k_n}^{p^n}=g_0 x$, there exists a unique morphism 
$\widetilde{\varphi}:\mathbb Q_p\rightarrow G$ such that, for any $n$,
$\widetilde{\varphi}(p^{-k_n})=g_{k_n}$ and 
$\widetilde{\varphi}$ clearly lifts $\varphi$.
\end{proof}

The second corollary is an improvement of Proposition \ref{proxxsxr}.

\bc
\label{corxxsxr}
Let $G$ be a weakly regular $p$-adic Lie group. One has the equality~:
$\g g_G=\{ X\in \g g_u \mid {\rm ad} X$ is nilpotent$\}$,
where 
$\g g_u$ is the Lie algebra of $G_u$.
\ec 

\begin{proof} This follows from Proposition \ref{proxxsxr} since,
by Proposition \ref{prolevdec}, one has the equality $\g g_u=\g s_u\oplus\g r_u$.
\end{proof}

The last corollary tells us that a weakly regular $p$-adic Lie group 
$G$ with $G=G_u$ is ``almost'' an algebraic Lie group.

\bc
\label{corgualg}
Let $G$ be a weakly regular $p$-adic Lie group such that $G=G_u$. 
Then there exists a Lie group morphism $\ps:G\ra H$ with finite kernel
and cokernel where $H$ is the group of $\m Q_p$-points of a
linear  algebraic group  defined over $\m Q_p$. 
\ec 

\begin{proof} According to Proposition \ref{proxxsxr}, 
$G=G_u$ is a semidirect product 
$G_u=S_u\ltimes R_u$. 
We choose $H$ to be the semi direct product $H:=S'\ltimes R_u$
where $S'$ is the Zariski closure of the group ${\rm Ad}(S_u)$ in
${\rm Aut}(\g g_u)$. Note that, since $G$ is weakly regular, 
any automorphism of $\g g_u$ induces an automorphism of $R_u$.
We define the morphism $\ps:G_u\ra H$  by 
$\psi(g)=({\rm Ad}(s),r)$ for $g=sr$ with $s\in S_u$, $r\in R_u$.
Proposition \ref{proxxsxr} tells us also that this morphism $\psi$ has finite kernel and cokernel.
\end{proof}

\subsection{Regular semiconnected component}
\label{secregsemcom}
\bq
We are now ready to prove the following theorem which was the main motivation of our paper.
\eq

\bt
\label{thwearegreg}
Let $G$ be a weakly regular $p$-adic Lie group.
Then, there exists an open regular subgroup $G_\Om$ of $G$ 
which contains all the one-parameter subgroups of $G$.
\et

\begin{Rmq}
Let $\Om$ be a {\it standard} open subgroup of $G$.
We define
the {\it $\Om$-semiconnected component of $G$}
as its open subgroup
$G_\Om:=\Om G_{u}$ (see \cite{Ra3}). 
In this language, Theorem \ref{thwearegreg} states that, 
the $\Om$-semiconnected component of a weakly regular $p$-adic Lie group 
is regular, if the standard subgroup $\Om$ is small enough.
\end{Rmq}

\begin{proof}[Proof of Theorem \ref{thwearegreg}]
We will need some notations. 
Let $\g s$ be a Levi subalgebra of $\g g$,
$\g s_u:=\g s\cap \g g_{nc}$, $\g s'$ the centralizer of 
$\g s_u$ in $\g s$, $\g r$  the radical of $\g g$, and 
$\g r'$ the centralizer of $\g s_u$ in $\g r$.

We can choose a standard subgroup $\Om'_S$ of $G$ with Lie algebra $\g s'$,
and a standard subgroup $\Om'_R$ of $G$ with Lie algebra $\g r'$
such that $\Om'_S$ normalizes $\Om'_R$
and the semidirect product $\Om':=\Om'_S\Om'_R$ is a standard subgroup
of $G$ with Lie algebra $\g s'\oplus \g r'$.
Since $G$ is weakly regular, by shrinking $\Om'$, we can assume 
that it
commutes with $S_u$ and normalizes $R_u$. 

We claim that,  if $\Om'$ is 
small enough, the group 
$$
G_\Om:=\Om'G_u
$$ 
is an open regular subgroup of $G$.
\vs

{\bf First step}~: Openness. 
One has the equalities $\g s=\g s'\oplus \g s_u$ and, 
according to Proposition \ref{proruru}, $\g r=\g r'+\g r_u$. 
Hence $G_\Om$ is open in $G$.
\vs

{\bf Second step}~: {\rm Ad}-regularity. 
Let $J\subset G_\Om$ be the kernel of ${\rm Ad}_{\g g}$.
We want to prove that $J$ is  the center of $G_\Om$.
Since $J$ acts trivially on the quotient $\g g/\g r\simeq \g s'\oplus\g s_u$, 
by Proposition \ref{prolevdec}, one has the inclusion
$J\subset J':=Z_E\Om'_RR_u$ where 
$Z_E$ is a finite subgroup of $S_u$. One has $J_u=J\cap R_u$ and this group
is algebraic unipotent. Hence,  by Lemma \ref{lemngq}, 
the quotient $R_u/J_u$ is also algebraic unipotent. 
By Lemma \ref{lemunipro},  the group $R_u/J_u$ is closed in 
the group $J'/J$. This means that $R_uJ$ is closed in $J'$, hence 
that $J/J_u$ is closed in the compact group $J'/R_u$. In particular,
$J/J_u$ is compact. Therefore, 
there exists a compact set $K\subset J$ such that 
$$
J=KJ_u. 
$$ 
We can now prove that if $\Om'$ is small enough the 
group $J$ commutes with $G_\Om$. This is a consequence of the following three
facts.

$(i)$ Since $J$ is the kernel of ${\rm Ad}_{\g g}$ and $G$ is weakly regular,
$J$ commutes with $G_u$.

$(ii)$ Since $J_u$ is a subgroup of $R_u$ 
whose Lie algebra $\g j_u$ is included in the center of $\g g$,
if we choose $\Om'$ small enough, one has ${\rm Ad}_{\g j_u}(\Om')=\{ e\}$,
and the group $J_u$ commutes with $\Om'$.

$(iii)$ Since $K$ is compact and ${\rm Ad}_{\g g}(K)=\{ e\}$,
by Lemma \ref{lemexpmap}, if we choose $\Om'$ small enough,
the group $K$ commutes with $\Om'$. 
\vs

{\bf Third step}~: Size of finite subgroups. 
We want a uniform upper bound on the cardinality of the finite 
subgroups of $G_\Om$. This follows from the inclusions 
$R_u\subset G_u\subset G_\Om$ of normal subgroups and from
the following three facts.

$(i)$ Since  the group $R_u$ is algebraic unipotent, it does not contain finite groups.

$(ii)$ Since, by Proposition \ref{prolevdec},  the group $G_u/R_u$ is a finite
extension of a linear group, by  Example \ref{exalinreg}.$b$,
its finite subgroups have bounded cardinality. 

$(iii)$ Since  the group $G_\Om/G_u$ is a compact $p$-adic Lie group,
by Example \ref{exalinreg}.$a$,
its finite subgroups have bounded cardinality. 
\end{proof}

\subsection{Non weakly regular $p$-adic Lie groups}
\label{secnonregsem}
\bq
Not every $p$-adic Lie group is weakly regular. 
Here is a surprising example.
\eq
\begin{Ex}
\label{exagguhhu}
There exists a $p$-adic Lie group $G$ with $G=G_u$ 
which does not contain any open weakly regular subgroup
$H$ with $H=H_u$.
\end{Ex}

We will give the construction of such a group $G$
with Lie algebra $\g g={\g s\g l}(2,\m Q_p)$, 
but we will leave the verifications to the reader.

We recall that the group $G_0:={\rm SL}(2,\m Q_p)$
is an amalgamated product $G_0=K \star_{I_0}K$
where $K:={\rm SL}(2,\m Z_p)$ and 
$I_0:=\{ k\in K\mid k_{21}\equiv 0\; {\rm mod}\; p\}
$ is an
Iwahori subgroup of $K$. 
We define $G$ as the amalgamated product 
$G=K \star_{I}K$ where $I\subset I_0$ is an open subgroup
such that $I\neq I_0$.

The morphism $G\ra G_0$ is a non central extension. 
Using the construction in Lemma \ref{lemconstrunipo} 
one can check that $G$ is spanned by one-parameter subgroups.
However, one can check that the universal central extension 
$\wt G_0\ra G_0$ can not be lifted as a morphism $\wt G_0\ra G$.

\newpage



\begin{thebibliography}{1}


\bibitem{BQII}Y. Benoist, J.-F. Quint,
Stationary measures and
invariant subsets of homogeneous spaces (II),
{\em Journ. Am. Math. Soc.} {\bf 26} (2013) 659-734.

\bibitem{BQLS}
Y. Benoist, J.-F. Quint,
Lattices in $S$-adic Lie groups,
{\em Journal of Lie Theory}  {\bf 24} (2014) 179-197.

\bibitem{BoTi65}
A. Borel, J. Tits,
Groupes r\'eductifs,
{\em Publ. Math. IHES} {\bf 27} (1965) 55-150.

\bibitem{BoTi73}
A. Borel, J. Tits,
Homomorphismes ``abstraits'' de groupes alg\'{e}briques simples,
{\em Ann. of Math.} {\bf 97} (1973) 499-571.
	
\bibitem{Bou1}
N. Bourbaki,
Groupes et Alg\`ebres de Lie, chapitre 1,
{\em CCLS} (1971).

\bibitem{BT1} F. Bruhat, J. Tits
{\em Groupes r\'eductifs sur un corps local I},
Publ. IHES 41 (1972) p.5-252.

\bibitem{BT2} F. Bruhat, J. Tits
{\em Groupes r\'eductifs sur un corps local II},
Publ. IHES 60 (1984) p.5-184.

\bibitem{DSMS} J. Dixon, M. duSautoy, A. Mann, D. Segal,
Analytic Pro-p Groups,
{\em CUP}  (1991).

\bibitem{La65} S. Lang, Algebra, Addison-Wesley (1964). 

\bibitem{MaTo} G. Margulis, G. Tomanov,
Invariant measures for actions of unipotent groups over local fields 
on homogeneous spaces,
{\em Invent. Math.} {\bf 116}  (1994) 347-392.

\bibitem{PlRa94} V. Platonov, A.Rapinchuk,
Algebraic groups and number theory. Academic Press (1994). 

\bibitem{PR84} G. Prasad, M. Raghunathan,
Topological central extensions of semi-simple groups over local fields,
{\em Ann. of Math.} {\bf 119} (1984) 143-268.



\bibitem{Ra3} M. Ratner,
Raghunathan's  conjectures for cartesian products
of real and p-adic Lie groups,
{\em Duke Math. J.} {\bf 72}
(1995) 275-382.

\bibitem{Ser72} J.P. Serre, Arbres, amalgames, SL2,
{\em Asterisque} {\bf 46} (1972).

\bibitem{Tits79} J. Tits, Reductive groups over local fields, {\em
Pr. Sym. P. Math.} {\bf 33}
(1979) 22-69.



\end{thebibliography}
\end{document}